\documentclass[a4paper, 12pt]{amsart}
\usepackage[cp1251]{inputenc}
\usepackage[english]{babel}
\usepackage{amsmath,amssymb,a4wide,latexsym,amscd,amsthm}
\usepackage{graphics}

\theoremstyle{plain}

\newtheorem{lemma}{Lemma}[section]
\newtheorem{proposition}{Proposition}[section]
\newtheorem{corollary}{Corollary}[section]

\newtheorem*{main}{Main Theorem}
\newtheorem*{A}{Theorem A}

\theoremstyle{definition}
\newtheorem{example}{Example}[section]

\newtheorem{remark}{Remark}[section]

\numberwithin{equation}{section}

\begin{document}

\title[]
{On the closedness of the sum of ranges of operators $A_k$ with almost compact products $A_i^* A_j$}

\author{Ivan S. Feshchenko}

\address{Taras Shevchenko National University of Kyiv,
Faculty of Mechanics and Mathematics, Kyiv, Ukraine}

\email{ivanmath007@gmail.com}



\begin{abstract}
Let $\mathcal{H}_1,\ldots,\mathcal{H}_n,\mathcal{H}$ be complex Hilbert spaces and
$A_k:\mathcal{H}_k\to\mathcal{H}$ be a bounded linear operator with the closed range $Ran(A_k)$, $k=1,\ldots,n$.
It is known that if $A_i^*A_j$ is compact for any $i\neq j$, then $\sum_{k=1}^n Ran(A_k)$ is closed.

We show that if all products $A_i^*A_j$, $i\neq j$ are "almost" compact, then the subspaces
$Ran(A_1),\ldots,Ran(A_n)$ are essentially linearly independent and their sum is closed.
\end{abstract}

\maketitle

\section{Introduction}

\subsection{On the closedness of the sum of subspaces}

Let $\mathcal{X}$ be a complex Hilbert space,
and $\mathcal{X}_1,\ldots,\mathcal{X}_n$ be subspaces of $\mathcal{X}$ (by a subspace we always mean a closed linear set).
Define the sum of $\mathcal{X}_1,\ldots,\mathcal{X}_n$ in the natural way, namely,
\begin{equation*}
\sum_{k=1}^n \mathcal{X}_k=\left\{\sum_{k=1}^n x_k\mid x_k\in \mathcal{X}_k,\quad k=1,\ldots,n\right\}.
\end{equation*}
We are interested in the following question:
\begin{center}
\textbf{when $\sum_{k=1}^n \mathcal{X}_k$ is closed?}
\end{center}
If $\mathcal{X}$ if finite dimensional, then, clearly, $\sum_{k=1}^n \mathcal{X}_k$ is closed.
However, in infinite dimensional space $\mathcal{X}$ this is not true (generally speaking).

\begin{example}
Let $\mathcal{Y},\mathcal{Z}$ be Hilbert spaces,
and $A:\mathcal{Y}\to\mathcal{Z}$ be a bounded linear operator with nonclosed $Ran(A)$.
Set $\mathcal{X}=\mathcal{Y}\oplus\mathcal{Z}$.
Define the subspaces
\begin{equation*}
\mathcal{X}_1=\mathcal{Y}\oplus 0=\{(y,0)\mid y\in \mathcal{Y}\},\quad
\mathcal{X}_2=Graph(A)=\{(y,Ay)\mid y\in\mathcal{Y}\}.
\end{equation*}
Then $\mathcal{X}_1+\mathcal{X}_2=\mathcal{Y}\oplus Ran(A)$ is not closed.
\end{example}

Systems of subspaces $\mathcal{X}_1,\ldots,\mathcal{X}_n$ for which the question
\begin{center}
\textbf{is $\sum_{k=1}^n \mathcal{X}_k$ closed?}
\end{center}
is very important arise in various branches of mathematics, for example, in
\begin{enumerate}
\item
theoretical tomography and theory of ridge functions (plane waves).
See, e.g., \cite{Boman84, Petersen79, SheppKruskal78, Pinkus97}, where
the problem on closedness of the sum of spaces of ridge functions (plane waves) is studied;
\item
statistics.
See, e.g., \cite{Bickel91}, where
the closedness of the sum of two marginal subspaces is important for
constructing an efficient estimation of linear functionals of a probability measure with known marginal distributions;
\item
projection algorithms for solving convex feasibility problems
(problems of finding a point in the nonempty intersection of $n$ closed convex sets).
See, e.g., \cite[Theorem 5.19]{BauschkeBorwein96},~\cite[Theorem 4.1]{Badea12},~\cite{Pustylnik12} and references therein;
\item
theory of operator algebras.
See, e.g., \cite{Hartz12}, where
the closedness of finite sums of full Fock spaces over subspaces of $\mathbb{C}^d$ plays a crucial role
for construction of a topological isomorphism between universal operator algebras;
\item
quadratic programming.
See, e.g., \cite{Schochetman01};
\end{enumerate}
and others.

\subsection{On the closedness of the sum of ranges of operators $A_k$ with almost compact products $A_i^* A_j$}

Let $\mathcal{H}_1,\ldots,\mathcal{H}_n,\mathcal{H}$ be complex Hilbert spaces and
$A_k:\mathcal{H}_k\to\mathcal{H}$ be a bounded linear operator with the closed range $Ran(A_k)$, $k=1,\ldots,n$.
We study a question on the closedness of
\begin{equation*}
\sum_{k=1}^n Ran(A_k)=\left\{\sum_{k=1}^n y_k\mid y_k\in Ran(A_k)\right\}=\left\{\sum_{k=1}^n A_k x_k\mid x_k\in\mathcal{H}_k\right\}
\end{equation*}
under additional conditions imposed on $A_k$, $k=1,\ldots,n$.
The following result follows immediately from~\cite[Theorem 2]{Feshchenko14}.

\begin{A}
If $A_i^* A_j$ is compact for any $i\neq j$, then $\sum_{k=1}^n Ran(A_k)$ is closed.
\end{A}

We will show that if the essential norms of $A_i^*A_j$, $i\neq j$, are "small enough" (in some sense), then
the subspaces $Ran(A_1),\ldots,Ran(A_n)$ are essentially linearly independent and their sum $\sum_{k=1}^n Ran(A_k)$ is closed.

To formulate our main result we need
\begin{enumerate}
\item
to introduce a notion of the essential linear independence of a system of subspaces;
\item
to introduce a notion of the essential reduced minimum modulus of an operator.
\end{enumerate}
It is by using of the essential reduced minimum modulus of $A_k$, $k=1,\ldots,n$ that
we will specify the meaning of the fuzzy words "small enough".

\subsection{Structure of the paper}
In Section~\ref{S:auxiliary notions} we introduce auxiliary notions, namely,
the essential reduced minimum modulus of an operator and the essential linear independence of subspaces.

In Section~\ref{S:main result and example} we formulate our main result (see the Main Theorem)
and consider an example of its application.

In Section~\ref{S:auxiliary lemmas} we prove two auxiliary lemmas.

Finally, in Section~\ref{S:proof main theorem} we prove the Main Theorem.

\subsection{Notation}

In this paper we consider only complex Hilbert spaces usually denoted by the letters $\mathcal{X},\mathcal{H},\mathcal{K}$.
The scalar product in a Hilbert space is denoted by $\langle\cdot,\cdot\rangle$,
and $\|\cdot\|$ stands for the corresponding norm, $\|x\|^2=\langle x,x\rangle$.
The identity operator on~$\mathcal{H}$ is denoted by~$I_\mathcal{H}$ or simply $I$ if it is clear which Hilbert space is being considered.
For a bounded linear operator $A:\mathcal{H}\to\mathcal{H}$,
$\sigma(A)$ denotes the spectrum of the operator $A$,
and $\sigma_e(A)$ denotes the essential spectrum of $A$.
For a bounded linear operator $A:\mathcal{H}_1\to\mathcal{H}_2$,
\begin{equation*}
\|A\|_e=\inf\{\|A+K\|\mid K:\mathcal{H}_1\to\mathcal{H}_2 \quad\text{is compact}\}
\end{equation*}
is the essential norm of $A$.

\section{Auxiliary notions: the essential reduced minimum modulus and the essential linear independence}
\label{S:auxiliary notions}

\subsection{The essential reduced minimum modulus}

First, let us recall the notion of the \emph{reduced minimum modulus} of an operator.
Let $\mathcal{H}_1,\mathcal{H}_2$ be Hilbert spaces,
$A:\mathcal{H}_1\to\mathcal{H}_2$ be a bounded linear operator, $A\neq 0$.
The reduced minimum modulus of $A$ is defined by
\begin{equation*}
\gamma(A)=\inf\{\|Ax\|\mid x\in\mathcal{H}_1\ominus Ker(A), \|x\|=1\}.
\end{equation*}
In other words, $\gamma(A)$ is the maximum of all $\gamma\geqslant 0$ such that
\begin{equation}\label{eq:rmm 1}
\|Ax\|\geqslant\gamma\|x\|
\end{equation}
for all $x\in\mathcal{H}_1\ominus Ker(A)$.

\begin{remark}
For the zero operator we set $\gamma(0)=+\infty$.
\end{remark}

The reduced minimum modulus possesses the following important property: $Ran(A)$ is closed if and only if $\gamma(A)>0$.

Indeed, if $\gamma(A)>0$, then $A(\mathcal{H}_1\ominus Ker(A))=A(\mathcal{H}_1)=Ran(A)$ is closed.
Conversely, suppose $Ran(A)$ is closed.
Define an operator
\begin{equation*}
A'=A\upharpoonright_{\mathcal{H}_1\ominus Ker(A)}:\mathcal{H}_1\ominus Ker(A)\to Ran(A).
\end{equation*}
Then $A'$ is an invertible operator between two Hilbert spaces.
Consequently, there exists $\gamma>0$ such that $\|A'x\|\geqslant\gamma\|x\|$, $x\in\mathcal{H}_1\ominus Ker(A)$.
This means that $\|Ax\|\geqslant\gamma\|x\|$, $x\in\mathcal{H}_1\ominus Ker(A)$ $\Rightarrow$ $\gamma(A)\geqslant\gamma>0$.

Note that~\eqref{eq:rmm 1} is equivalent to $\|Ax\|^2\geqslant\gamma^2\|x\|^2$.
This inequality can be rewritten as
\begin{equation}\label{eq:rmm 2}
\langle A^*Ax,x\rangle\geqslant\gamma^2\|x\|^2.
\end{equation}
Define a self-adjoint operator
\begin{equation*}
B=A^*A\upharpoonright_{\mathcal{H}_1\ominus Ker(A)}:\mathcal{H}_1\ominus Ker(A)\to\mathcal{H}_1\ominus Ker(A).
\end{equation*}
(Thus, $A^*A=B\oplus 0$ with respect to the orthogonal decomposition $\mathcal{H}=(\mathcal{H}_1\ominus Ker(A))\oplus Ker(A)$.)
Now, from~\eqref{eq:rmm 2} we see that $\gamma(A)$ is the maximum of all $\gamma\geqslant 0$ such that
\begin{equation*}
B\geqslant\gamma^2 I.
\end{equation*}

Now we are ready to introduce, in a natural way, the \emph{essential reduced minimum modulus} of an operator.
The essential reduced minimum modulus of $A$, $\gamma_e(A)$, is the supremum of all $\gamma\geqslant 0$ for which
there exists a compact self-adjoint operator $K:\mathcal{H}_1\ominus Ker(A)\to\mathcal{H}_1\ominus Ker(A)$ such that
\begin{equation*}
B+K\geqslant\gamma^2 I.
\end{equation*}

\begin{remark}
For the zero operator we set $\gamma_e(0)=+\infty$.
\end{remark}

Since the notion of the essential reduced minimum modulus plays a crucial role in this paper,
we present some properties and formulas for $\gamma_e(A)$.

\textbf{1.}
Clearly, $\gamma_e(A)\geqslant\gamma(A)$.
Consequently, if $Ran(A)$ is closed, then $\gamma_e(A)\geqslant\gamma(A)>0$.

\textbf{2.1.}
If $\mathcal{H}_1\ominus Ker(A)$ is finite dimensional, then $\gamma_e(A)=+\infty$.

\textbf{2.2.}
Suppose $\mathcal{H}_1\ominus Ker(A)$ is infinite dimensional.
Then
\begin{equation}\label{eq:ermm 1}
\gamma_e(A)=(\min\{\lambda\mid\lambda\in\sigma_e(B)\})^{1/2}.
\end{equation}
This follows from the following simple proposition.

\begin{proposition}
Let $C:\mathcal{H}\to\mathcal{H}$ be a bounded self-adjoint operator in an infinite dimensional Hilbert space $\mathcal{H}$.
Define $m_e(C)$ to be the supremum of all $m$ for which there exists a compact self-adjoint operator $K:\mathcal{H}\to\mathcal{H}$ such that
\begin{equation*}
C+K\geqslant mI.
\end{equation*}
Then $m_e(C)=\min\{\lambda\mid\lambda\in\sigma_e(C)\}$.
\end{proposition}

Thus,
\begin{equation*}
(\gamma_e(A))^2=m_e(B)=\min\{\lambda\mid\lambda\in\sigma_e(B)\};
\end{equation*}
it follows~\eqref{eq:ermm 1}.

Now suppose that $Ran(A)$ is closed.
Then $\gamma_e(A)>0$ and $\sigma_e(B)\subset[(\gamma_e(A))^2,+\infty)$.
Hence, from~\eqref{eq:ermm 1} it follows that
\begin{equation*}
\gamma_e(A)=(\min\{\lambda\mid\lambda\in\sigma_e(A^*A)\setminus\{0\}\})^{1/2}.
\end{equation*}

\textbf{3.}
If for $\gamma\geqslant 0$ there exists a compact operator $T:\mathcal{H}_1\ominus Ker(A)\to\mathcal{H}_2$ such that
\begin{equation}\label{eq:ermm 2}
\|Ax\|^2+\|Tx\|^2\geqslant\gamma^2\|x\|^2
\end{equation}
for all $x\in\mathcal{H}_1\ominus Ker(A)$,
then $\gamma_e(A)\geqslant\gamma$.
Indeed,~\eqref{eq:ermm 2} can be rewritten as
\begin{equation*}
\langle (A^*A+T^*T)x,x\rangle\geqslant\gamma^2\|x\|^2;
\end{equation*}
hence,
\begin{equation*}
B+T^*T\geqslant\gamma^2 I.
\end{equation*}
Since $T^*T$ is compact and self-adjoint, we conclude that $\gamma_e(A)\geqslant\gamma$.

\begin{remark}
Suppose that $\dim\mathcal{H}_1\leqslant\dim\mathcal{H}_2$.
(Here $\dim\mathcal{H}$ is the Hilbert dimension of $\mathcal{H}$, i.e.,
$\dim\mathcal{H}$ is the cardinality of an orthonormal basis of $\mathcal{H}$.
This inequality holds in the most interesting case when $\mathcal{H}_1,\mathcal{H}_2$ are separable and infinite dimensional.)
Then $\gamma_e(A)$ is the supremum of all $\gamma\geqslant 0$
for which there exists a compact operator $T:\mathcal{H}_1\ominus Ker(A)\to\mathcal{H}_2$
such that~\eqref{eq:ermm 2} holds for all $x\in\mathcal{H}_1\ominus Ker(A)$.

This follows from the arguments above and the fact that every compact self-adjoint operator $K:\mathcal{H}_1\ominus Ker(A)\to\mathcal{H}_1\ominus Ker(A)$
can be represented as $K=T^*T$ for some compact $T:\mathcal{H}_1\ominus Ker(A)\to\mathcal{H}_2$.
\end{remark}

\subsection{An essential linear independence}

Let $\mathcal{X}$ be a Hilbert space, $\mathcal{X}_1,\ldots,\mathcal{X}_n$ be its subspaces.
We say that $\mathcal{X}_1,\ldots,\mathcal{X}_n$ are linearly independent if the equality
\begin{equation*}
\sum_{i=1}^n x_i=0,
\end{equation*}
where $x_i\in\mathcal{X}_i$, $i=1,\ldots,n$ implies that $x_i=0$, $i=1,\ldots,n$.
Clearly, $\mathcal{X}_1,\ldots,\mathcal{X}_n$ are linearly independent if and only if
\begin{equation*}
\mathcal{X}_i\cap\sum_{j\neq i}\mathcal{X}_j=\{0\}
\end{equation*}
for $i=1,\ldots,n$.

Now, it is natural to say that $\mathcal{X}_1,\ldots,\mathcal{X}_n$ are essentially linearly independent if
the linear set $\mathcal{X}_i\cap\sum_{j\neq i}\mathcal{X}_j$ is finite dimensional for $i=1,\ldots,n$.

\begin{remark}
It is useful to reformulate the properties of linear independence and essential linear independence it terms of properties of
the operator $S:\mathcal{X}_1\oplus\ldots\oplus\mathcal{X}_n\to\mathcal{X}$ defined by
\begin{equation*}
S(x_1,\ldots,x_n)=\sum_{i=1}^n x_i,\quad x_i\in\mathcal{X}_i,\quad i=1,\ldots,n.
\end{equation*}
Clearly,
\begin{enumerate}
\item
$\mathcal{X}_1,\ldots,\mathcal{X}_n$ are linearly independent $\Leftrightarrow$ $Ker(S)=\{0\}$;
\item
$\mathcal{X}_1,\ldots,\mathcal{X}_n$ are essentially linearly independent $\Leftrightarrow$ $Ker(S)$ is finite dimensional.
\end{enumerate}
\end{remark}

The essentially linearly independent systems of subspaces can be regarded as
a finite dimensional perturbation of the linearly independent systems of subspaces.
More precisely, we have
\begin{enumerate}
\item
if $\mathcal{X}_i=\mathcal{Y}_i+\mathcal{Z}_i$, $i=1,\ldots,n$, where
subspaces $\mathcal{Y}_1,\ldots,\mathcal{Y}_n$ are linearly independent and
subspaces $\mathcal{Z}_1,\ldots,\mathcal{Z}_n$ are finite dimensional, then
$\mathcal{X}_1,\ldots,\mathcal{X}_n$ are essentially linearly independent;
\item
if $\mathcal{X}_1,\ldots,\mathcal{X}_n$ are essentially linearly independent, then
there exists a representation
$\mathcal{X}_i=\mathcal{Y}_i\oplus\mathcal{Z}_i$, $i=1,\ldots,n$, where
the subspaces $\mathcal{Y}_1,\ldots,\mathcal{Y}_n$ are linearly independent and
the subspaces $\mathcal{Z}_1,\ldots,\mathcal{Z}_n$ are finite dimensional.
\end{enumerate}

To prove (1), note that
\begin{equation*}
\dim \left(\mathcal{X}_i\cap\sum_{j\neq i}\mathcal{X}_j\right)\leqslant\dim\mathcal{Z}_i+\dim\sum_{j\neq i}\mathcal{Z}_j.
\end{equation*}

To prove (2), it is sufficient to define
\begin{equation*}
\mathcal{Z}_i=\mathcal{X}_i\cap\sum_{j\neq i}\mathcal{X}_j,\quad
\mathcal{Y}_i=\mathcal{X}_i\ominus\mathcal{Z}_i
\end{equation*}
for $i=1,\ldots,n$.

\section{The main result and an example of its application}\label{S:main result and example}

\subsection{The main result}

Let $\mathcal{H}_1,\ldots,\mathcal{H}_n,\mathcal{H}$ be complex Hilbert spaces and
$A_k:\mathcal{H}_k\to\mathcal{H}$ be a bounded linear operator with the closed range $Ran(A_k)$, $k=1,\ldots,n$.

Suppose that
\begin{enumerate}
\item
$\gamma_e(A_k)\geqslant\gamma_k>0$, $k=1,\ldots,n$;
\item
$\|A_i^*A_j\|_e\leqslant\varepsilon_{i,j}$ for $i\neq j$.
\end{enumerate}
In what follows, we assume that $\varepsilon_{i,j}=\varepsilon_{j,i}$ for all $i\neq j$.
Define a real symmetric $n\times n$ matrix $M=(m_{i,j})$ by
\begin{equation*}
m_{i,j}=\begin{cases}
\gamma_i^2,   &\text{if $i=j$;}\\
-\varepsilon_{i,j},&\text{if $i\neq j$.}
\end{cases}
\end{equation*}

\begin{main}
If $M$ is positive definite, then $Ran(A_1),\ldots,Ran(A_n)$ are essentially linearly independent and their sum is closed.
\end{main}

Note, that if
\begin{equation}\label{eq:diagonally dominant}
\sum_{j\neq i}\varepsilon_{i,j}<\gamma_i^2
\end{equation}
for $i=1,\ldots,n$, then $M$ is a strictly diagonally dominant matrix;
consequently, $M$ is positive definite.
Hence, we get the following corollary of the Main Theorem.

\begin{corollary}
If~\eqref{eq:diagonally dominant} holds for $i=1,\ldots,n$, then
$Ran(A_1),\ldots,Ran(A_n)$ are essentially linearly independent and their sum is closed.
\end{corollary}

\subsection{Example}

Let $\mathcal{X}$ be a Hilbert space, and $\mathcal{X}_1,\ldots,\mathcal{X}_n$ be its subspaces.
Using the Main Theorem, we will obtain sufficient conditions under which
$\mathcal{X}_1,\ldots,\mathcal{X}_n$ are essentially linearly independent and their sum is closed.

For a subspace $\mathcal{Y}$ of $\mathcal{X}$, define $P_{\mathcal{Y}}$ to be the orthogonal projection onto $\mathcal{Y}$.
Clearly, $Ran(P_{\mathcal{Y}})=\mathcal{Y}$ and
\begin{equation*}
\gamma_e(P_{\mathcal{Y}})=
\begin{cases}
1,      &\text{if $\mathcal{Y}$ is infinite dimensional;}\\
+\infty,&\text{if $\mathcal{Y}$ is finite dimensional.}
\end{cases}
\end{equation*}
We apply Main Theorem to the operators $P_{\mathcal{X}_k}:\mathcal{X}\to\mathcal{X}$.
Suppose that numbers $\varepsilon_{i,j}=\varepsilon_{j,i}$, $i\neq j$ are such that
$\|P_{\mathcal{X}_i}P_{\mathcal{X}_j}\|_e\leqslant\varepsilon_{i,j}$ for $i\neq j$.
Define a real symmetric $n\times n$ matrix $M=(m_{i,j})$ by
\begin{equation*}
m_{i,j}=\begin{cases}
1,   &\text{if $i=j$;}\\
-\varepsilon_{i,j},&\text{if $i\neq j$.}
\end{cases}
\end{equation*}

By the Main Theorem, if $M$ is positive definite, then
$\mathcal{X}_1,\ldots,\mathcal{X}_n$ are essentially linearly independent and their sum is closed.
In particular, if $\sum_{j\neq i}\varepsilon_{i,j}<1$ for any $i=1,\ldots,n$, then
$\mathcal{X}_1,\ldots,\mathcal{X}_n$ are essentially linearly independent and their sum is closed.

\section{Auxiliary lemmas}\label{S:auxiliary lemmas}

\subsection{On the closedness of the sum of operator ranges}

Let $\mathcal{H}_1,\ldots,\mathcal{H}_n,\mathcal{H}$ be Hilbert spaces,
$A_k:\mathcal{H}_k\to\mathcal{H}$ be a bounded linear operator, $k=1,\ldots,n$.

\begin{lemma}\label{L:closed sum operator ranges}
If there exists $\varepsilon>0$ such that
\begin{equation*}
\sigma(\sum_{k=1}^n A_k A_k^*)\cap(0,\varepsilon)=\varnothing
\end{equation*}
then $\sum_{k=1}^n Ran(A_k)$ is closed.
\end{lemma}

\begin{remark}
The converse is also true.
However, this fact is not needed in the paper.
\end{remark}

\begin{proof}
Define $\mathcal{R}=\sum_{k=1}^n Ran(A_k)$ and consider the orthogonal decomposition
\begin{equation*}
\mathcal{H}=\overline{\mathcal{R}}\oplus(\mathcal{H}\ominus\overline{\mathcal{R}}).
\end{equation*}
With respect to this orthogonal decomposition
\begin{equation*}
A_k=\begin{pmatrix}
A_k'\\
0
\end{pmatrix},
\end{equation*}
where $A_k':\mathcal{H}_k\to\overline{\mathcal{R}}$, $A_k'$ is
$A_k$ considered as the operator from $\mathcal{H}_k$ to $\overline{\mathcal{R}}$.
Then $A_k^*=((A_k')^*\quad 0)$.
Hence, $\sum_{k=1}^n A_k A_k^*=(\sum_{k=1}^n A_k'(A_k')^*)\oplus 0$.
We have
\begin{align*}
&Ker(\sum_{k=1}^n A_k'(A_k')^*)=\bigcap_{k=1}^n Ker((A_k')^*)=\\
&=\bigcap_{k=1}^n\{x\in\overline{\mathcal{R}}\mid x\bot Ran(A_k')\}=\bigcap_{k=1}^n\{x\in\overline{\mathcal{R}}\mid x\bot Ran(A_k)\}=\\
&=\{x\in\overline{\mathcal{R}}\mid x\bot\mathcal{R}\}=\{0\}.
\end{align*}
Moreover, $\sigma(\sum_{k=1}^n A_k'(A_k')^*)\cap(0,\varepsilon)=\varnothing$.
Conclusion: $\sum_{k=1}^n A_k'(A_k')^*\geqslant\varepsilon I$.
It follows that $Ran(\sum_{k=1}^n A_k'(A_k')^*)=\overline{\mathcal{R}}$.
Consequently, $\sum_{k=1}^n Ran(A_k')=\overline{\mathcal{R}}$.
Hence, $\sum_{k=1}^n Ran(A_k)=\overline{\mathcal{R}}$ is closed as required.
\end{proof}

\subsection{On the essential spectrum of block operators}

Let $\mathcal{H}_1,\ldots,\mathcal{H}_n$ be Hilbert spaces,
$A:\mathcal{H}_1\oplus\ldots\oplus\mathcal{H}_n\to\mathcal{H}_1\oplus\ldots\oplus\mathcal{H}_n$ be a bounded self-adjoint operator.
Let
\begin{equation*}
A=(A_{i,j}\mid i,j=1,\ldots,n)
\end{equation*}
be the block decomposition of $A$.
Suppose that reals $a_i$, $i=1,\ldots,n$ and $a_{i,j}=a_{j,i}$, $i\neq j$, satisfy the following conditions:
\begin{enumerate}
\item
$m_e(A_{i,i})\geqslant a_i$ for $i=1,\ldots,n$;
\item
$\|A_{i,j}\|_e\leqslant a_{i,j}$ for any $i\neq j$.
\end{enumerate}

(Recall that for a bounded self-adjoint operator $C$,
$m_e(C)$ is the supremum of all $m$ for which there exists a compact self-adjoint operator $K$ such that $C+K\geqslant mI$.)

Define a real symmetric $n\times n$ matrix $M=(m_{i,j})$ by
\begin{equation*}
m_{i,j}=\begin{cases}
a_i,   &\text{if $i=j$;}\\
-a_{i,j},&\text{if $i\neq j$.}
\end{cases}
\end{equation*}

\begin{lemma}\label{L:ess spectrum block operator}
$\sigma_e(A)\subset[\lambda_{min}(M),+\infty)$, where
$\lambda_{min}(M)$ is the minimum eigenvalue of $M$.
\end{lemma}
\begin{proof}
First, let us show that for any $\varepsilon>0$ there exists a compact self-adjoint operator $K$ such that
\begin{equation}\label{eq:ineq A+compact}
A+K\geqslant (\lambda_{min}(M)-\varepsilon)I
\end{equation}
There exist compact self-adjoint operators $K_{i,i}$, $i=1,\ldots,n$,
and compact operators $K_{i,j}$, $i\neq j$, such that
\begin{enumerate}
\item
$A_{i,i}+K_{i,i}\geqslant(a_i-\varepsilon/n)I$ for $i=1,\ldots,n$;
\item
$\|A_{i,j}+K_{i,j}\|\leqslant a_{i,j}+\varepsilon/n$ for $i\neq j$.
\end{enumerate}
Clearly, we can assume that $K_{j,i}=K_{i,j}^*$ for $i\neq j$.
Set
\begin{equation*}
K=(K_{i,j}\mid i,j=1,\ldots,n).
\end{equation*}
For any $v=(v_1,\ldots,v_n)\in\mathcal{H}_1\oplus\ldots\oplus\mathcal{H}_n$ we have
\begin{align*}
&\langle (A+K)v,v\rangle=\sum_{i=1}^n\langle (A_{i,i}+K_{i,i})v_i,v_i\rangle+2\sum_{i<j}Re\langle (A_{i,j}+K_{i,j})v_j,v_i\rangle\geqslant\\
&\geqslant\sum_{i=1}^n(a_i-\varepsilon/n)\|v_i\|^2-2\sum_{i<j}(a_{i,j}+\varepsilon/n)\|v_j\|\|v_i\|\geqslant\\
&\geqslant\langle M(\|v_1\|,\ldots,\|v_n\|)^T,(\|v_1\|,\ldots,\|v_n\|)^T\rangle-\varepsilon\sum_{i=1}^n\|v_i\|^2\geqslant\\
&\geqslant (\lambda_{min}(M)-\varepsilon)\|v\|^2
\end{align*}
It follows~\eqref{eq:ineq A+compact}.

Now we are ready to prove the assertion of the lemma.
Consider any $\varepsilon>0$.
There exists a compact self-adjoint operator $K$ such that~\eqref{eq:ineq A+compact} holds.
By the Weyl theorem, $\sigma_e(A)=\sigma_e(A+K)\subset[\lambda_{min}(M)-\varepsilon,+\infty)$.
Since $\varepsilon>0$ is arbitrary, we conclude that $\sigma_e(A)\subset[\lambda_{min}(M),+\infty)$.
\end{proof}

\section{Proof of the Main Theorem}\label{S:proof main theorem}

Set $\mathcal{K}_i=\mathcal{H}_i\ominus Ker(A_i)$, $i=1,\ldots,n$.
Define an operator $\Gamma:\mathcal{K}_1\oplus\ldots\oplus\mathcal{K}_n\to\mathcal{H}$ by
\begin{equation*}
\Gamma(x_1,\ldots,x_n)=\sum_{i=1}^n A_i x_i,\quad x_i\in\mathcal{K}_i,\quad i=1,\ldots,n.
\end{equation*}
Then $\Gamma^*:\mathcal{H}\to\mathcal{K}_1\oplus\ldots\oplus\mathcal{K}_n$ and
\begin{equation*}
\Gamma^*x=(A_1^*x,\ldots,A_n^*x),\quad x\in\mathcal{H}.
\end{equation*}
Hence, $\Gamma\Gamma^*=\sum_{i=1}^n A_i A_i^*$.
Consider the operator
\begin{equation*}
G=\Gamma^*\Gamma:\mathcal{K}_1\oplus\ldots\oplus\mathcal{K}_n\to\mathcal{K}_1\oplus\ldots\oplus\mathcal{K}_n.
\end{equation*}
Its block decomposition is the following:
\begin{equation*}
G=(A_i^* A_j\upharpoonright_{\mathcal{K}_j}:\mathcal{K}_j\to\mathcal{K}_i\mid i,j=1,\ldots,n).
\end{equation*}
By Lemma~\ref{L:ess spectrum block operator} and the positive definiteness of $M$, we have $0\notin\sigma_e(G)$.

Let us show that $Ran(A_1),\ldots,Ran(A_n)$ are essentially linearly independent.
Since $0\notin\sigma_e(G)$, we conclude that $Ker(G)$ is finite dimensional.
We have $Ker(G)=Ker(\Gamma^*\Gamma)=Ker(\Gamma)$.
Hence,
\begin{equation*}
Ker(\Gamma)=\{(x_1,\ldots,x_n)\in\mathcal{K}_1\oplus\ldots\oplus\mathcal{K}_n\mid\sum_{i=1}^n A_i x_i=0\}
\end{equation*}
is finite dimensional.
It follows that $Ran(A_1),\ldots,Ran(A_n)$ are essentially linearly independent.

Let us show that $\sum_{i=1}^n Ran(A_i)$ is closed.
Since $0\notin\sigma_e(G)$, we conclude that $\sigma(G)\cap(0,\varepsilon)=\varnothing$ for some $\varepsilon>0$.
Since $\sigma(\Gamma\Gamma^*)\setminus\{0\}=\sigma(\Gamma^*\Gamma)\setminus\{0\}$, we see that
$\sigma(\sum_{i=1}^n A_i A_i^*)\cap(0,\varepsilon)=\varnothing$.
By Lemma~\ref{L:closed sum operator ranges}, $\sum_{i=1}^n Ran(A_i)$ is closed.

The proof is complete.

\end{document}